\documentclass[11pt]{amsart}

\usepackage{mathptmx} % rm & math
\usepackage[scaled=0.90]{helvet} % ss
\usepackage{courier} % tt
\normalfont
\usepackage[T1]{fontenc}
\usepackage{color}

\usepackage{hyperref}
\usepackage{amsmath}
\usepackage{amsthm}
\usepackage{amsfonts}
\usepackage{amssymb}
\usepackage{graphicx}
\DeclareGraphicsExtensions{.eps}

\newtheorem{theorem}{Theorem}[section]
\newtheorem{utv*}{Proposition}
\newtheorem{hyp*}{Conjecture}
\newtheorem{lemma}[theorem]{Lemma}

\newtheorem{defin}{Definition}
\newtheorem{zamech}{Remark}
\newtheorem*{th*}{Theorem}

\newcommand{\av}[2]{\langle #1\rangle_{_{\scriptstyle #2}}}
\newcommand{\ave}[1]{\langle #1\rangle}

\def\sli{\sum\limits}
\def\ili{\int\limits}

\def\a{\alpha}
\def\R{\mathbb{R}}

\def\ep{\varepsilon}

\def\E{\mathbb{E}}
\def\cD{\mathcal{D}}

\newcommand{\al}{\alpha}

\newcommand{\diam}{\operatorname{diam}}

\newcommand{\cz}{Calder\'{o}n--Zygmund\ }

\newcommand{\QQ}{[w]_{A_2}}

\newcommand{\wt}{\widetilde}
\newcommand{\La}{\langle}
\newcommand{\Ra}{\rangle}

\newcommand{\cP}{\mathcal{P}}

\newcommand{\cE}{\mathcal{E}}

\def\cyr{\fontencoding{OT2}\fontfamily{wncyr}\selectfont}
\DeclareTextFontCommand{\textcyr}{\cyr}
\newcommand{\sha}[0]{\ensuremath{\mathbb{S}%H\!\!S
}}

{\end{list}}

\newcounter{vremennyj}

\begin{document}

\title{The proof of $A_2$ conjecture in a geometrically doubling metric space}
\author{Fedor Nazarov}
\address{Department of Mathematics, University of Wisconsin-Madison and Kent Sate University}
\email{nazarov@math.wisc.edu}
\thanks{Work of F.~Nazarov is supported  by the NSF grant  DMS-0758552}

\author{Alexander Reznikov}
\address{Department of Mathematics,  Michigan State University, East
Lansing, MI 48824, USA}
\email{reznikov@ymail.com}
\author{Alexander Volberg}
\thanks{Work of A.~Volberg is supported  by the NSF under the grant  DMS-0758552.
}
\address{Department of Mathematics, Michigan State University, East
Lansing, MI 48824, USA}
\email{volberg@math.msu.edu}
\urladdr{http://sashavolberg.wordpress.com}

%\maketitle

\makeatletter
\@namedef{subjclassname@2010}{
  \textup{2010} Mathematics Subject Classification}
\makeatother


% 42B	Harmonic analysis in several variables
% 42B20	Singular and oscillatory integrals (Calder?on-Zygmund, etc.)
% 42B35	Function spaces arising in harmonic analysis

% 47A	General theory of linear operators
% 47A30	Norms (inequalities, more than one norm, etc.)

%{30E20, 47B37, 47B40, 30D55.}
%
% 30D55	$H^p$-classes (1980-2009)
% 30E20	Integration, integrals of Cauchy type, integral representations of analytic functions
%
% 47B   	Special classes of linear operators
% 47B37	Operators on special spaces (weighted shifts, operators on sequence spaces, etc.)
% 47B40	Spectral operators, decomposable operators, well-bounded operators, etc.


   \begin{abstract}
We give a proof of the $A_2$ conjecture in geometrically doubling metric spaces (GDMS), i.e. a metric space where one can  fit not more than a fixed amount of disjoint balls of radius $r$ in a ball of radius $2r$. Our proof consists of three main parts: a construction of a random ``dyadic'' lattice in a metric space; a clever averaging trick from ~\cite{Hyt}, which decomposes a ``hard'' part of a Calder\'on-Zygmund operator into dyadic shifts (adjusted to metric setting); and the estimates for these dyadic shifts, made in ~\cite{NV} and later in ~\cite{Tr}.
\end{abstract}

\date{}
\maketitle

\section{Introduction}
\label{Intro}

Recall that in \cite{PTV1} it was proved that
\begin{theorem}
\label{weak}
If $T$ is an arbitrary operator with a Calder\'on--Zygmund kernel, then
\begin{align*}
\|T\|_{L^2(w)\rightarrow L^{2,\infty}(w)}+ & \|T'\|_{L^2(w^{-1})\rightarrow L^{2,\infty}(w^{-1})}
 \le 2\|T\|_{L^2(w)\rightarrow L^2(w)}
\\
& \le C\,([w]_{A_2}+ \|T\|_{L^2(w)\rightarrow L^{2,\infty}(w)}+ \|T'\|_{L^2(w^{-1})\rightarrow L^{2,\infty}(w^{-1})}).
\end{align*}
\end{theorem}
By $T'$ we denote the adjoint operator. Here of course only the right inequality is interesting. And it is unexpected too. The weak and strong norm of any operator with a Calder\'on--Zygmund kernel turned out to be equivalent up to additive term $[w]_{A_2}$. From this we obtained in \cite{PTV1} the result which holds for any Calder\'on--Zygmund operator.
\begin{theorem}
\label{A2log}
$\|T\|_{L^2(w)\rightarrow L^2(w)} \le C\cdot [w]_{A_2}\log (1+[w]_{A_2})$.
\end{theorem}
By $A_2$ conjecture people understand the strengthening of this claim, where the logarithmic term is deleted, in other words, a linear  (in weight's norm) estimate of  arbitrary weighted Calder\'on--Zygmund operator.
In \cite{HLRSVUT} the $A_2$ conjecture was proved for Calder\'on--Zygmund operators having more than $2d$ smoothness in $\mathbb{R}^d$. The $A_2$ conjecture was fully proved in a preprint of T. Hyt\"onen, see \cite{Hyt}. The proof is based on the main theorem in the paper  \cite{PTV1} of P\'erez--Treil--Volberg.  Both \cite{PTV1} and \cite {Hyt} are neither short nor easy.

 The direct proof of $A_2$ conjecture (without going through \cite{PTV1}) was given in \cite{HPTV}, and it was based on two ingredients: 1) a formula for decomposing an arbitrary Calder\'on--Zygmund operators into (generalized) dyadic shifts by the averaging trick, 2) on a polynomial in complexity and linear in weight estimate of the norm of a dyadic shift.

 The latter was quite complicated and was based on modification of the argument in Lacey--Petermichl-Reguera \cite{LPR}. The former was rooted in the works on non-homogeneous Harmonic Analysis, like e. g. \cite{NTV}-- \cite{NTVlost}, but with  a new twist, which appeared first in Hyt\"onen's \cite{Hyt} and was simplified in  Hyt\"onen--P\'erez--Treil--Volberg's \cite{HPTV}.

 The averaging trick was a development of the bootstrapping argument used by Nazarov--Treil--Volberg \cite{NTV}--\cite{NTVlost}, where they exploited the fact that the bad part of a function can be made arbitrarily small. Using the original Nazarov--Treil--Volberg averaging trick would add an extra factor depending on $[w]_{A_2}$ to the estimate, so a new idea was necessary. The new observation in \cite{Hyt} was that as soon as the probability of a ``bad'' cube is less than $1$, it is possible to completely ignore the bad cubes (at least in the situation where they cause troubles).

\subsection{Structure of the paper}
Here we give a proof of the $A_2$ conjecture in geometrically doubling metric spaces (GDMS), i.e. a metric space where one can  fit not more than a fixed amount of disjoint balls of radius $r$ in a ball of radius $2r$. 

The paper is organized as follows:

\begin{enumerate}
\item A construction of a probability space of random ``dyadic'' lattice in a metric space is given in Section \ref{randlatt};
\item Averaging trick of Hyt\"onen \cite{Hyt} (but we think we simplified it) is given in Section \ref{decomp}; 
\item A linear estimate of weighted dyadic shift on metric space from \cite{NV}, which uses Bellman function technique, is given in Sections \ref{para} and \ref{shiftest}.
 For another proof of the linear estimate for weighted dyadic shifts, which can be easily adjusted to the metric case, we refer to ~\cite{Tr}.
\end{enumerate}

Our main result is the following.
\begin{theorem}[$A_2$ theorem for a geometrically doubling metric space]
Let $X$ be a geometrically doubling metric space, $\mu$ and $T$ as above, $w\in A_{2, \mu}$.
In addition we assume that $\mu$ is a doubling measure. Then
\begin{equation}\label{mainest}
\|T\|_{L^2(wd\mu)\to L^2(wd\mu)} \leqslant C(T)[w]_{2, \mu}.
\end{equation}
\end{theorem}
We postpone precise definitions to the Section \ref{A2theorem}. The precise definition of a geometrically doubling metric space is given in the next section.

\section{First step}\label{randlatt}
Consider a  compact doubling metric space $X$ with metric $d$ and doubling constant $A$. Instead of $d(x,y)$ we write $|xy|$.
Precisely, the definition is the following.
\begin{defin}
Suppose $(X, |.|)$ is a metric space. We call it geometrically doubling with constant $A$, if for any $x\in X$ and $r>0$ we can fit no more than $A$ disjoint balls of radius $r/2$ in the ball $B(x, r)$. 
\end{defin}

As authors of ~\cite{HM}, we essentially use the idea of Michael Christ ~\cite{Chr}, but randomize his construction in a different way. Therefore, we want to guard the reader that even though on the surface the proof below is very close to the proof from ~\cite{HM}, however, our construction is essentially different, and so the proof of the assertion in our main lemma, which was not hard in ~\cite{HM}, becomes much more subtle here.

The main difference between the construction \cite{HM} and here is that the one here is of ``bottom to top'' type, meaning that the centers of ``father cubes'' are chosen randomly, after the centers of ``son cubes'' are fixed. The construction in ~\cite{HM} goes ``top to bottom'', and it is not that clear to us why ``father cubes'' have enough independence from ``son cubes'' to ensure that in the model where elementary event is {\bf one} dyadic lattice, the event for a cube of a lattice to be ``bad'' (see the definition below) with respect to cubes of the {\bf same} lattice is strictly less then one. However, we still feel that the construction of \cite{HM} can most probably be used for the purposes of our result as well, we  just feel that it is  a bit more easy to follow that everything  falls in its place with our construction below.

We now proceed to the construction.

For a number $k>0$ we say that a set $G$ is a $k$-grid if $G$ is maximal (with respect to inclusion) set, such that for any $x,y\in G$ we have $d(x,y)> k$.

Let from now on $\diam X=1$. Take a small positive number $\delta \ll 1$ depending on the doubling constant of $X$ and a large natural number $N$, and for every $M\geqslant N$ fix $G_M=\{z_M^\a\}$,  a certain $\delta^M$-grid of $X$. Now take $G_N$ and randomly choose a $G_{N-1}=\delta^{N-1}$-grid in $G_N$. Then take $G_{N-1}$ and randomly choose a $G_{N-2}=\delta^{N-2}$-grid in $G_{N-1}$.  Do this $N$ times. Notice that $G_0$ consists of just one random point of $G_N$.

We explain what is ``randomly''. Since $X$ is a compact metric space, all $G_k$'s are finite. Therefore, there are finitely many $(N-1)$-grids in $G_N$. We choose one of them with a probability 
$$
\frac{1}{\mbox{number of $(N-1)$-grids in $G_N$}}.$$

Our first lemma is the following.
\begin{lemma}
\label{lm21}
For $k=0,\dots, N$
$$
\bigcup\limits_{y\in G_{N-k}} B(y, 3\delta^{N-k}) = X.
$$
\end{lemma}

\begin{zamech} For $N+k, k\ge 0,$  instead of $N-k$ this is obvious.
\end{zamech}

\begin{proof}
Take $x\in X$. %$x\in G_{N-k}\subset G_{N-k+1}\subset \ldots \subset G_N$.
Then, since $G_{N}$ is maximal, there exists a point $y_0 \in G_{N}$, such that $|xy_0|\leqslant \delta^{N}$. Since $G_{N-1}$ is maximal in $G_N$, there is a point $y_1\in G_{N-1}$, such that $|y_0y_1|\leqslant \delta^{N-1}$. Similarly we get $y_2, \ldots, y_k$ and then
$$
|xy_k| \leqslant |xy_0|+\ldots + |xy_k|\leqslant \delta^{N}+\ldots + \delta^{N-k} = \delta^{N-k}(1+\delta+\ldots + \delta^{k})\leqslant \frac{\delta^{N-k}}{1-\delta} \leqslant 2\delta^{N-k}.
$$
\end{proof}

Once we have all our sets $G_N$, we introduce a relationship $\prec$ between points. We follow ~\cite{HM} and ~\cite{Chr}.

Take a point $y_{k+1}\in G_{k+1}$. There exists at most one $y_k\in G_k$, such that $|y_{k+1}y_k|\leqslant \frac{\delta^{k}}{4}$. This is true since if there are two such points $y_k^1, \; y_k^2$, then
$$
|y_k^1 y_k^2|\leqslant \frac{\delta^{k}}{2},
$$
which is a contradiction, since $G_k$ was a $\delta^{k}$-grid in $G_{k+1}$.

Also there exists at least one $z_k\in G_k$ such that $|y_{k+1} z_k|\leqslant 3 \delta^{k}$. This is true by the lemma.

Now, if there exists an $y_k$ as above, we set $y_{k+1}\prec y_k$. If no, then we pick one of $z_k$ as above and set $y_{k+1}\prec z_k$. For all other $x\in G_{k}$ we set $y_{k+1}\not\prec x$. Then extend by transitivity.

We also assume that $y_k\prec y_k$. This is if $y_k$ on the left happened to belong already to $G_{k+1}$.

We do this procedure randomly and independently, and treat same families of $G_k$'s with different $\prec$-law as different families.

Take now a point $y_k\in G_k$ and define
$$
Q_{y_k} = \bigcup\limits_{z\prec y_k, z\in G_{\ell}} B(z, \frac{\delta^{\ell}}{100}).
$$

\begin{lemma}
\label{cover}
For every $k$ we have
$$
X=\bigcup\limits_{y_k\in G_k} \textup{clos}(Q_{y_k})
$$
\end{lemma}
\begin{zamech} There is only one point in $G_0$,  and ${clos}(Q_{y}), y\in G_0,$ is just $X$. But for small $\delta$, $X=\bigcup\limits_{y_1\in G_1} \textup{clos}(Q_{y_1})$ is a genuine (and random) splitting of $X$.
\end{zamech}

\begin{proof}
Take any $x\in X$. By the previous lemma, for every $m> k$ there exists a point $x_m\in G_m$, such that $|xx_m|\leqslant 3\delta^{m}$. In particular, $x_m \to x$. Fix for a moment $x_m$. Then there are points $y_{m-1}\in G_{m-1}, \ldots, y_k\in G_{k}$, such that $x_m \prec y_{m-1}\prec \ldots \prec y_k$. In particular, $x_m \in Q_{y_k}$, where $y_k$ depends on $x_m$. Then
$$
|y_k x| \leqslant |y_k x_m| + |x_m x|\leqslant |y_k x_m| + 3\delta^{m} \leqslant |y_k x_m| + 3\delta^{k}.
$$
Moreover, by the chain of $\prec$'s, we know that $|y_k x_m|\leqslant 10\delta^{k}$. Therefore,
$$
|y_k x|\leqslant 15\delta^{k}.
$$
We claim that the set $\{y_k\}=\{y_k(x_m)\}_{m\geqslant k}$ is bounded independently of $m$. This is true since all $y_k$'s are separated from each other and by the doubling of our space (we are ``stuffing'' the ball $B(x, 15\delta^{k})$ with balls $B(y_k, \delta^{k})$).

So, take an infinite subsequence $x_{m}$ that corresponds to one point $y_k\in G_k$. Then we get $x_m \in Q_{y_k}$, $x_m\to x$, so $x\in \textup{clos} Q_{y_k}$, and we are done.
\end{proof}
\begin{zamech}
Since the space $X$ is compact, our random procedure consists of finitely many steps. Therefore, our probability space is discreet. We suggest to think about all probabilities just as number of good events divided by number of all events.

However, all our estimates will not depend on number of steps (and, therefore, diameter of $X$), which is essential.
\end{zamech}
\begin{zamech}
We notice that in the Euclidian space, say, $\R$, this procedure does not give a standard dyadic lattice.
\end{zamech}
\section{Second step: technical lemmata}
Define
$$
\tilde{Q}_{y_k} = X\setminus \bigcup_{z_k\not= y_k, z_k\in G_k}\textup{clos}\,Q_{z_k}.
$$
In particular,
$$
Q_{y_k}\subset \tilde{Q}_{y_k}\subset \textup{clos}(Q_{y_k}).
$$
\begin{lemma}[Lemma 4.5 in ~\cite{HM}]
\label{l45}
Let $m$ be a natural number, $\ep>0$, and $\delta^m \geqslant 100\ep$. Suppose $x\in \textup{clos}\,Q_{y_k}$ and $dist(x, X\setminus \tilde{Q}_{y_k})<\ep \delta^{k}$. Then for any chain
$$
z_{k+m}\prec z_{k+m-1}\prec\ldots \prec z_{k+1}\prec z_k,
$$
such that $x\in \textup{clos}\,Q_{z_{k+m}}$, the following relationships hold
$$
|z_i z_j|\geqslant \frac{\delta^{j}}{100}, \; \; \; k\leqslant j < i \leqslant k+m.
$$

\end{lemma}
\begin{proof}
Suppose $|z_i z_j| < \frac{\delta^{j}}{100}$. We first consider a case when $z_k=y_k$.
Since $z_j\prec z_k=y_k$, we have $B(z_j, \frac{\delta^j}{200})\subset Q_{y_k}\subset \tilde{Q}_{y_k}$. Therefore,
$$
\frac{\delta^j}{200}\leqslant dist(z_j, X\setminus \tilde{Q}_{y_k})\leqslant dist(x, X\setminus \tilde{Q}_{y_k}) + dist(x, z_i) + dist(z_i, z_j) < \ep \delta^{k} + 5\delta^i + \frac{\delta^{j}}{100}
$$
If $\delta$ is less than, say, $1\over 1000$, then we get a contradiction.

The only not obvious estimate is that $dist(x, z_i)<5\delta_i$. It is true since $x\in \textup{clos}\,Q_{z_{k+m}}$.

We have proved the lemma with assumption that $z_k=y_k$. Let us get rid of this assumption. We know that
$$
x\in \textup{clos}\,Q_{z_{k+m}}\subset \textup{clos}\,Q_{z_{k}}.
$$
Also we have $x\in \textup{clos}\,Q_{y_k}$, so, since
$$
\tilde{Q}_{z_k}=X\setminus \bigcup_{u_k\not = z_k} \textup{clos}\,Q_{u_k} \subset X\setminus \textup{clos}\,Q_{y_k},
$$
we get $x\in X\setminus \tilde{Q}_{z_k}$. In particular, $dist(x, X\setminus \tilde{Q}_{z_k})=0<\ep \delta^k$, and we are in the situation of the first part. This finishes our proof.
\end{proof}

\begin{lemma}
\label{MAIN}
Fix $x_k\in G_k$. Then
\begin{equation}
\label{ryadom}
\mathbb{P}(\exists x_{k-1}\in G_{k-1}\colon |x_k x_{k-1}|<\frac{\delta^{k-1}}{1000})\geqslant a
\end{equation}
for some $a\in (0,1)$.
\end{lemma}
\begin{proof}
We remind that we are in a compact metric situation. By rescaling we can think that we work with $G_1$ and choose $G_0$. We can even think that the metric space consists of finitely many points, it is $X:=G_2$. The finite set $G_1\subset X$ consists of points  having the following properties:

\noindent 1. $\forall x,y\in G_1$ we have $|xy|\ge \delta$;

\noindent 2. if $z\in X\setminus G_1$ then $\exists x\in G_1$ such that $|zx|<\delta$.

These two properties are equivalent to saying that the subset $G_1$ of $X$ consists of points such that $\forall x,y\in G_1$ we have $|xy|\ge \delta$ and we cannot add any point from $X$ to $G_1$ without violating that property. In other words: $G_1$ is a {\it maximal} set with property 1.

Recall that here the word ``maximal'' means maximal with respect to inclusion, not maximal in the sense of the number of elements.

Now we consider the new metric space $Y=G_1$ and $G_0$ is any maximal subset such that
\begin{equation}
\label{max1}
\forall x, y\in G_0\,,\, |xy|\ge 1\,.
\end{equation}
In other words, we have
\noindent 1. $\forall x,y\in G_0$ we have $|xy|\ge 1$;

\noindent 2. if $z\in Y\setminus G_0$ then $\exists x\in G_0$ such that $|zx|<1$.

There are finitely many such maximal subsets $G_0$ of $Y$. We prescribe for each choice the same probability.
Now we want to prove the claim that is even stronger than \eqref{ryadom}. Namely, we are going to prove that given $y\in Y$

\begin{equation}
\label{ryadom1}
\mathbb{P}(\exists x_{0}\in G_{0}\colon x_0 =y)\geqslant a\,,
\end{equation}
where $a$ depends only on $\delta$ and the constants of geometric doubling of our compact metric space.

Let $Y$ be any metric space with finitely many elements. We will color the points of $Y$ into red and green colors. The coloring is called proper if

\noindent 1. every red point does not have any other red point at distance $<1$;

\noindent 2. every green point has at least one red point at distance $<1$.

Given {\it a proper coloring} of $Y$ the collection of red points is called $1$-{\it lattice}. It is a maximal (by inclusion) collection of points at distance $\ge 1$ from each other.

What we need to finish the proof is

\begin{lemma}
\label{finiteY}
Let $Y$ be a finite metric space as above. Assume $Y$ has the following property:
\begin{equation}
\label{finite}
\text{In every ball of radius  less than}\,\, 1\,\,\text{ there are at most}\,\, d\,\,\text{ elements}\,.
\end{equation}
 Let $\mathcal{L}$ be a collection of $1$-lattices in $Y$. Elements of $\mathcal{L}$ are called $L$. Let $v\in Y$. Then
$$
\frac{\text{the number of 1-lattices L such that v belongs to L}}{\text{the total number of 1-lattices L}} \ge a>0\,,
$$
where $a$ depends only on $d$.
\end{lemma}

\begin{proof}
Given $v\in Y$ consider all subsets of $B(v,1)\setminus{v}$, this collection is called $\mathcal{S}$. Let $S\in \mathcal{S}$. We call $W_S$ the collection of all proper colorings such that $v$ is green, all elements of $S$ are red, and all elements of $B(v,1) \setminus S$ are green. We call $\tilde S$ all points in $Y$, which are not in $B(v,1)$, but at distance $<1$ from some point in $S$.

All proper colorings of $Y$ such that $v$ is red are called $B$. Let us show that
\begin{equation}
\label{2d}
\text{card}\, W_S\le \text{card}\, B\,.
\end{equation}
Notice that if \eqref{2d} were proved, we would be done with Lemma \ref{finiteY}, $a\ge 2^{-d+1}$, and, consequently, the proof of the main lemma would be finished, $a\ge 2^{-\delta^{-D}}$, where $D$ is a geometric doubling constant.

To prove \eqref{2d} let us show that we can recolor any proper coloring from $W_S$ into the one from $B$, and that this map is injective. Let $L\in W_S$. We

\noindent 1. Color $v$ into red;

\noindent 2. Color $S$ into green;

\noindent 3. Elements of $\tilde S$ were all green before. We leave them green, but we find among them all those $y$ that now in the open ball $B(y,1)$  in $Y$ all elements are green. We call them yellow (temporarily) and denote them
$Z$;

\noindent 4. We enumerate $Z$ in any way (non-uniqueness is here, but we do not care);

\noindent 5. In the order of enumeration color yellow points to red, ensuring that we skip recoloring of a point in $Z$ if it is at $<1$  distance to any previously colored yellow-to-red point from $Z$. After several steps all green and yellow elements of $\tilde S$ will have the property that at distance $<1$ there is a red point;

\noindent 6. Color the rest of yellow (if any) into green and stop.

We result in a proper coloring (it is easy to check), which is obviously $B$. Suppose $L_1, L_2$ are two different proper coloring in $W_S$. Notice that the colors of $v, S, B(v,1)\setminus S$, $\tilde S$ are the same for them. So they differ somewhere else. But our procedure does not touch ``somewhere else". So the modified colorings $L_1', L_2'$ that we obtain after the algorithm 1-6 will differ as well may be even more). So our map $W_S\rightarrow B$ (being not uniquely defined) is however injective. We proved \eqref{2d}.

\end{proof}

Thus, the proof of the Lemma \ref{MAIN} is finished.

\end{proof}

\noindent{\bf Remark.} We are grateful to Michael Shapiro and Dapeng Zhan who helped us to prove Lemma  \ref{MAIN}.

\section{Main definition and theorem}
Fix a number $\gamma$, $0<\gamma<1$. Later the choice of $\gamma$ will be dictated by the Calder\'{o}n-Zygmund properties of the operator $T$. Also fix a sufficiently big $r$. The choice of $r$ will be made in this section.

\begin{defin}[Bad cubes]
Take a ``cube'' $Q=Q_{x_k}$. We say that $Q$ is good if there exists a cube $Q_1=Q_{x_n}$, such that
if
$$
\delta^k \leqslant \delta^{r} \delta^n \; \; \; (k\geqslant n+r)
$$
then either
$$
dist(Q, Q_1)\geqslant \delta^{k\gamma}\delta^{n(1-\gamma)}
$$
or
$$
dist(Q, X\setminus Q_1)\geqslant \delta^{k\gamma}\delta^{n(1-\gamma)}.
$$
\end{defin}
\begin{zamech}
Notice that $\delta^k=\ell(Q)$ just by definition.
\end{zamech}

If $Q$ is not good we call it bad.
\begin{theorem}
Fix a cube $Q_{x_k}$. Then
$$
\mathbb{P}(Q_{x_k}\; \mbox{is bad}\:) \leqslant \frac{1}{2}.
$$
\end{theorem}
\begin{zamech}[Discussion]
This theorem makes sense because when we fix a cube $Q_k$, say, $k\geqslant N$, so the grid $G_k$ is not even random, we can make big cubes random. And we claim that for big quantity of choices, our big cubes will have $Q_k$ either ``in the middle'' or far away, but not close to the boundary.
\end{zamech}

\begin{defin}
For $Q=Q_{x_k}$ define
$$
\delta_{Q}(\ep)=\delta_{Q}=\{x\colon dist(x, Q)\leqslant \ep\delta^{k} \;\mbox{and}\; dist(x, X\setminus Q)\leqslant \ep\delta^{k}\}
$$
\end{defin}

\begin{lemma}\label{sloy}
Let us start with level $N$ by fixing a $\delta^N$-grid (non-random), and let $k<N$, $x_k$ denoting the points  of the (random) grid $G_k$. Fix a point $x\in X$.
$$
\mathbb{P}(\exists x_k\in G_k:\,x\in \delta_{Q_{x_k}}) \leqslant \ep^{\eta}
$$
for some $\eta>0$.
\end{lemma}
\begin{proof}[Proof of the theorem]
Take the cube $Q_{x_k}$. There is a unique (random!) point $x_{k-s}$ such that $x_k\in Q_{x_{k-s}}$. Then
$$
dist(Q_{x_k}, X\setminus Q_{x_{k-s}})\geqslant dist(x_k, X\setminus Q_{x_{k-s}}) - diam(Q_{x_k}) \geqslant dist(x_k, X\setminus Q_{x_{k-s}}) - C \delta^{k}.
$$
Assume that $dist(x_k, X\setminus Q_{x_{k-s}})>2\delta^{k\gamma}\delta^{(k-s)(1-\gamma)}$ and that $s\geqslant r$ (this assumption is obvious, otherwise $Q_{x_{k-s}}$ does not affect goodness of $Q_{x_k}$).

Then, if $r$ is big enough ($\delta^{r(1-\gamma)}<\frac{1}{C}$) we get
$$
dist(Q_{x_k}, X\setminus Q_{x_{k-s}}) \geqslant \delta^{k\gamma}\delta^{(k-s)(1-\gamma)},
$$
and so $Q_{x_k}$ is good.
Therefore,
$$
\mathbb{P}(Q_{x_k} \; \mbox{is bad}\:) \leqslant C \sli_{s\geqslant r} \mathbb{P}(x_k\in \delta_{Q_{k-s}}(\ep=2\delta^{s\gamma}))\leqslant C \sli_{s\geqslant r} \delta^{\eta \gamma s}\leqslant 100C \delta^{\eta \gamma r}.
$$
By the choice of $\eta$, for sufficiently large $r$ this is less than $1\over 2$.
\end{proof}
%\begin{zamech}[Discussion]
%At the end of the proof we have claimed that
%$$
%\mathbb{P}(Q_{x_k} \; \mbox{is bad}\:) \leqslant C \sli_{s\geqslant r} \mathbb{P}(x_k\in \delta_{Q_{k-s}}(\ep=\delta^{s\gamma}))\leqslant C \sli_{s\geqslant r} \delta^{\eta \gamma s}\leqslant 100C \delta^{\eta \gamma r}.
%$$
%In particular, we did some estimate of the probability that $x_{k}\in \delta_{Q_{k-s}}$. Here it is crucial that the point $x_k$ is fixed and not random, so in some sense cubes $Q_{k-s}$ does not depend on $x_k$ (the conditional probability is equal to the unconditional probability, since $Q_k$ is fixed and not random).
%\end{zamech}

\begin{proof}[Proof of the lemma]
Let $x_k$ be such that $x\in \text{clos}\,Q_{x_k}$ (see Lemma \ref{cover}). We will estimate $\mathbb{P} ( dist(x, X\setminus \tilde{Q}_k)<\ep \delta^{k})\,|\, x\in \text{clos}\,Q_{x_k})$.
Fix the largest $m$ such that $500\ep \leqslant \delta^{m}$. Choose a point $x_{k+m}$ such that $x\in \text{clos}\,Q_{x_{k+m}}$. Then by the main lemma
$$
\mathbb{P}(\exists x_{k+m-1}\in G_{k+m-1}\colon |x_{k+m} x_{k+m-1}|<\frac{\delta^{k+m-1}}{1000})\geqslant a.
$$
Therefore,
$$
\mathbb{P}(\forall x_{k+m-1}\in G_{k+m-1}\colon |x_{k+m} x_{k+m-1}|\geqslant\frac{\delta^{k+m-1}}{1000})\leqslant 1-a.
$$
Let now
$$
x_{k+m}\prec x_{k+m-1}.
$$
Then
$$
\mathbb{P}(\forall x_{k+m-2}\in G_{k+m-2}\colon |x_{k+m-1} x_{k+m-2}|\geqslant\frac{\delta^{k+m-2}}{1000})\leqslant 1-a.
$$
So by Lemma \ref{l45}
$$
\mathbb{P}(dist(x, X\setminus \tilde{Q}_k)<\ep \delta^{k})\leqslant \mathbb{P}(|x_{k+j}x_{k+j-1}|\geqslant \frac{\delta^{k+j-1}}{1000} \; \forall j=1,\ldots, m) \leqslant (1-a)^{m}\leqslant C\ep^{\eta}
$$
for
$$
\eta=\frac{\log{(1-a)}}{\log(\delta)}.
$$
\end{proof}
\section{Probability to be ``good'' is the same for every cube}
We make the last step to make the probability to be ``good'' not just bounded away from zero, but the same for all cubes.
We use the idea from ~\cite{HM2}.

Take a cube $Q(\omega)$. Take a random variable $\xi_{Q}(\omega^{'})$, which is equally distributed on $[0,1]$. We know that
$$
\mathbb{P}(Q \; \mbox{is good}) = p_{Q}>a>0.
$$
We call $Q$ ``really good'' if
$$
\xi_Q \in [0, \frac{a}{p_Q}].
$$
Otherwise $Q$ joins bad cubes. Then
$$
\mathbb{P}(Q \; \mbox{is really good}) = a,
$$
and we are done.

\section{Application}\label{A2theorem}
As a main application of our construction, we state the following theorem.

\begin{defin}
Let $X$ be a geometrically doubling metric space.

Let $\lambda(x,r)$ be a positive function, increasing and doubling in $r$, i.e. $\lambda(x, 2r)\leqslant C\lambda(x,r)$, where $C$ does not depend on $x$ and $r$.

Suppose $K(x,y)\colon X\times X \to \R$ is a Calderon-Zygmund kernel, associated to a function $\lambda$, i. e.
\begin{align}
&|K(x,y)|\leqslant C \min\left( \frac{1}{\lambda(x,|xy|)}, \frac{1}{\lambda(y,|xy|)}\right),\\
& |K(x,y)-K(x',y)|\leqslant C \frac{|xx'|^\ep}{|xy|^\ep \lambda(x, |xy|)}, \; \;|xy|\geqslant C|xx'|,\\
&|K(x,y)-K(x,y')|\leqslant C \frac{|yy'|^\ep}{|xy|^\ep \lambda(y, |xy|)}, \; \;|xy|\geqslant C|yy'|.
\end{align}
By $B(x,r)$ we denote the ball in $|.|$ metric, i.e., $B(x,r)=\{y\in X\colon |yx|<r\}$.

Let $\mu$ be a measure on $X$, such that $\mu(B(x,r))\leqslant C\lambda(x,r)$, where $C$ does not depend on $x$ and $r$.
We say that $T$ is a Calderon-Zygmund operator with kernel $K$ if
\begin{align}
& T \; \mbox{is bounded} \; L^2(\mu)\to L^2(\mu), \\
& Tf(x) = \int K(x,y)f(y)d\mu(y), \; \forall x\not\in \textup{supp}\mu, \; \forall f\in C_0^\infty.
\end{align}
\end{defin}

\begin{defin}
Let $w>0$ $\mu$-a.e. Define
$$
w\in A_{2, \mu}\Leftrightarrow [w]_{2, \mu} = \sup_{x,r} \frac{1}{\mu(B(x,r))} \ili_{B(x,r)} wd\mu \cdot \frac{1}{\mu(B(x,r))} \ili_{B(x,r)} w^{-1}d\mu < \infty.
$$
\end{defin}
\begin{theorem}[$A_2$ theorem for a geometrically doubling metric space]
Let $X$ be a geometrically doubling metric space, $\mu$ and $T$ as above, $w\in A_{2, \mu}$.
In addition we assume that $\mu$ is a doubling measure. Then
\begin{equation}\label{mainest}
\|T\|_{L^2(wd\mu)\to L^2(wd\mu)} \leqslant C(T, X)[w]_{2, \mu}.
\end{equation}
\end{theorem}

\begin{zamech}
We note that existence of such $\mu$ on any GDMS was proved in ~\cite{KV}.
\end{zamech}
\subsection{Proof of the theorem}\label{decomp}
%Let us first make the following assumption: suppose $T\chi_X = T^{*}1=0$.

Take two step functions, $f$ and $g$. We first fix an $N$-grid $G_N$ in $X$, and ``cubes'' on level $N$, such that $f$ and $g$ are constants on every such cube. Then we start our randomization process.

As we mentioned, this process consists of finitely many steps, so all probabilistic terminology becomes trivial: we have a finite probability space.

Starting from $G_N$, we go ``up'' and on each level get dyadic cubes (random Christ's cubes). They have the usual structure of being either disjoint or one containing the other. For each dyadic cube $Q$ we have several dyadic sons, they are denoted by $s_i(Q)$, $i=1,\dots, M(Q)\le M$.  The number $M$ here is universal and depends only on geometric doubling constants of the space $X$. 

\begin{defin} By $\cE_k$ we denote set of all dyadic ``cubes'' of generation $k$. We call $Q_k^i\subset Q_{k-1}^j$, $Q_k^i\in \cE_k$  {\bf sons} of $Q_{k-1}^j$. \end{defin}

 With every cube $Q=Q_{x_{k}}$ we associate {\bf Haar functions} $h^j_{Q}$, $j=1,\dots, M-1$, with following properties:
\begin{enumerate}
\item $h^j_Q$ is supported on $Q$;
\item $h^j_Q$ takes constant values on each ``son'' of $Q$;
\item For any two cubes $Q$ and $R$, we have $(h^j_Q, h^i_R)=0$, and $(h^j_Q, 1)=0$;
\item $\|h^j_Q\|_{\infty}\leqslant \frac{C}{\sqrt{\mu(Q)}}$.
\end{enumerate}
We notice that the last property implies that $\|h^j_{Q}\|_{2}\leqslant C$.

We use angular brackets to denote the average: $\langle f \rangle_{Q, \mu} := \frac1{\mu(Q)}\int_Q f\, d\mu$. When we average over the whole space $X$, we drop the index and write $\langle f \rangle = \frac1{\mu(X)} \int_X fd\mu$.

\bigskip

Our main ``tool'' is going to be the famous ``dyadic shifts''. Precisely, we call by $\sha_{m,n}$ the operator given by the  kernel
$$
f\rightarrow \sum_{L\in \cD} \int_L a_L(x,y)f(y)dy\,,
$$
where
$$
a_L(x,y) =\sum_{\substack{ I\subset L, J\subset L\\ g(I)= g(L)+m, \, g(J)= g(L)+n}}c_{L,I,J} h_J^j(x)h_I^i(y)\,,
$$
where $h_I^i, h_J^j$ are Haar functions normalized in $L^2(d\mu)$ and satisfying (iv), and $|c_{L,I,J}|\le \frac{\sqrt{\mu(I)}\sqrt{\mu(J)}}{\mu(L)}$.
Often we will skip superscripts $i, j$.

\bigskip

Our next aim is to decompose the bilinear form of the operator $T$ into bilinear forms of dyadic shifts, which are estimated in the Section \ref{shiftest}. The rest will be the so-called ``paraproducts'', estimated in the Section \ref{para}.

Functions $\{\chi_X\}\cup\{h^j_Q\}$ form an orthogonal basis in the space $L^2(X,\mu)$. Therefore, we can write
$$
f=\ave{f}\chi_X + \sli_Q \sli_j(f, h^j_Q)h^j_Q, \;\; \; \; \; g=\ave{g}\chi_X+\sli_R\sli_i (g, h^i_R)h^i_R.
$$

First, we state and proof the theorem, that says that essential part of bilinear form of $T$ can be expressed in terms of pair of cubes, where the smallest one is good. We follow the idea of Hyt\"onen \cite{Hyt}. In fact, the work \cite{Hyt} improved on ``good-bad" decomposition of \cite{NTV}, \cite{NTV2}, \cite{NTV3} by replacing inequalities by an equality.

\begin{theorem}\label{rand1}
Let $T$ be any linear operator. Then the following equality holds:
$$
\pi_{good}\E \sli_{\stackrel{Q,R,i,j}{\ell(Q)\geqslant \ell(R)}}(Th_Q^j, h_R^i)(f, h_Q^j)(g, h_R^i) = \E \sli_{\stackrel{Q,R,i,j}{\ell(Q)\geqslant \ell(R), \; R \; \mbox{is good}}}(Th_Q^j, h_R^i)(f, h_Q^j)(g, h_R^i).
$$
The same is true if we replace $\geqslant$ by $>$.
\end{theorem}
\begin{proof}
We denote
$$
\sigma_1(T) = \sli_{\ell(Q)\geqslant \ell(R)}(Th^j_Q, h^i_R)(f, h^j_Q)(g, h^i_R).
$$
$$
\overline{\sigma_1(T)} = \sli_{\substack{\ell(Q)\geqslant \ell(R)\\ R \; \mbox{is good}}}(Th^j_Q, h^i_R)(f, h^j_Q)(g, h^i_R).
$$
We would like to get a relationship between $\E \sigma_1(T)$ and $\E\overline{\sigma_1(T)}$.

We fix $R$ and write (using $g_{good}:= \sli_{R \; \mbox{is good}} (g,h^i_R)h^i_R$)
$$
\sli_{Q} \sli_{R \; \mbox{is good}}(Th^j_{Q}, h^i_R)(f, h^j_Q)(g,h^i_R) = \left(T(f-\ave{f}\chi_X),\sli_{R \; \mbox{is good}} (g,h^i_R)h^i_R\right) =\left(T(f-\ave{f}\chi_X), g_{good}\right)\,.
$$
Taking expectations, we obtain

\begin{multline}
\label{allRgood}
\E \sli_{Q,R} (Th^j_{Q}, h^i_R)(f, h^j_Q)(g,h^i_R)\mathbf{1}_{R \,\mbox{is good}} =\\
 \E (T(f-\ave{f}\chi_X), g_{good})= (T(f-\ave{f}\chi_X),\E\, g_{good})=\\
 \pi_{good}(T(f-\ave{f}\chi_X),g)=\pi_{good}\E \sli_{Q,R} (Th^j_{Q}, h^i_R)(f, h^j_Q)(g,h^i_R).
\end{multline}
Next, suppose $\ell(Q)<\ell(R)$. Then goodness of $R$ does not depend on $Q$, and so
$$
 \pi_{good}  (Th^j_Q, h^i_R)(f, h^j_Q)(g, h^i_R)=\E \left((Th^j_Q, h^i_R)(f, h^j_Q)(g, h^i_R)\mathbf{1}_{R \,\mbox{is good}}|Q,R\right) \,.
$$
Let us explain this equality. The right hand side is conditioned: meaning that the left hand side involves the fraction of the number of all lattices containing $Q, R$ in this lattice and such that $R$ (the larger one) is good to the number of lattices containing $Q, R$ in it. This fraction is exactly $\pi_{good}$. Now we fix a pair of $Q, R$, $\ell(Q)<\ell(R)$, and multiply both sides by the probability that this pair is in the same dyadic lattice from our family. This probability is just the ratio of the number of dyadic lattices in our family containing elements $Q$ and $R$  to the number of all dyadic lattices in our family. After multiplication by this ratio and the summation of all terms with $\ell(Q)<\ell(R)$ we get
finally,
\begin{equation}
\label{QmRgood}
 \pi_{good}\E \sli_{\ell(Q)<\ell(R)} (Th^j_Q, h^i_R)(f, h^j_Q)(g, h^i_R)=\E \sli_{\ell(Q)<\ell(R)} (Th^j_Q, h^i_R)(f, h^j_Q)(g, h^i_R)\mathbf{1}_{R \,\mbox{is good}} \,.
 \end{equation}
 Now we use first \eqref{allRgood} and then \eqref{QmRgood}:
\begin{multline}
\pi_{good}\E \sli_{Q,R} (Th^j_{Q}, h^i_R)(f, h^j_Q)(g,h^i_R) = \E \sli_{Q,R} (Th^j_{Q}, h^i_R)(f, h^j_Q)(g,h^i_R)\mathbf{1}_{R \,\mbox{is good}} = \\=\E \sli_{\ell(Q)<\ell(R)}(Th^j_{Q}, h^i_R)(f, h^j_Q)(g,h^i_R)\mathbf{1}_{R \,\mbox{is good}} +
 \E\sli_{\ell(Q)\geqslant\ell(R)}(Th^j_{Q}, h^i_R)(f, h^j_Q)(g,h^i_R)\mathbf{1}_{R \,\mbox{is good}} = \\=\pi_{good}\E \sli_{\ell(Q)<\ell(R)}(Th^j_{Q}, h^i_R)(f, h^j_Q)(g,h^i_R) +
\E\sli_{\ell(Q)\geqslant\ell(R), R\;\mbox{is good}}(Th^j_{Q}, h^i_R)(f, h^j_Q)(g,h^i_R),
\end{multline}
and therefore
\begin{equation}
\label{QbRgood}
\E\sli_{\ell(Q)\geqslant\ell(R), R\;\mbox{is good}}(Th^j_{Q}, h^i_R)(f, h^j_Q)(g,h^i_R) = \pi_{good}\E\sli_{\ell(Q)\geqslant\ell(R)}(Th^j_{Q}, h^i_R)(f, h^j_Q)(g,h^i_R).
\end{equation}
\end{proof}
This is the main trick. To have the whole sum expressed as the multiple of the sum, where the {\bf smaller} in size cube is good, is {\bf very useful} as we will see. It gives extra decay on matrix coefficients $(Th^j_{Q}, h^i_R)$ and allows us to represent our operator as ``convex combination of  dyadic shifts".
%%%%%%%%%%
%\begin{zamech}[!!!!!!!!]
%Mne kazhetsya, chto nado dat' chyotkoe opredelenie uslovnih mat. ozhidaniy i t. p. U nas vsyo konechno, poetomu eto dolzhno bit' trivial'no. Napishete?
%
%Ili prosto napisat' tak:
%\begin{multline*}
%\E (...\cdot {\bf 1}_{R \; good}|Q, R) = \E (... \cdot {\bf 1_{R \;good}}{\bf 1_Q}{\bf 1_R}) = \E(...) \E{\bf 1_Q} \E{\bf{1_{R\; good} 1_R}} =\\= \E(...)\E{\bf 1_Q} \pi_{good}\E{\bf 1_R} = \pi_{good}\cdot \E(...\cdot {\bf 1_Q %1_R})?
%\end{multline*}
%\end{zamech}
%%%%%%%%%%%%%%

So, we have obtained that
$$
\E \sigma_1(T) = \pi_{good}^{-1} \cdot \E\overline{\sigma_1(T)}.
$$

Thus, to estimate $\E \sigma_1(T)$ it is enough to estimate $\E \overline{\sigma_1(T)}$.  Absolutely the same symmetrically holds for $\sigma_2(T)$.

\subsection{Paraproducts}
In this subsection we take care of the terms $\ave{f}\chi_X$ and $\ave{g}\chi_X$. These terms will lead to so called paraproducts. In fact, let us introduce three auxiliary operators:

\begin{align}
&\pi(f):=\pi_{T\chi_X}(f):=\sli_{Q, j}\av{f}{Q}(T\chi_X, h_Q^j)h_Q^j;\\
&\pi_*(f) := \sli_{Q, j}(f, h_Q^j)(T^*\chi_X, h_Q^j)\frac{\chi_Q}{\mu(Q)} = (\pi_{T^*\chi_X})^* (f);\\
&o(f):=\ave{f}\ave{T\chi_X}\chi_X.
\end{align}

Recall that $\ave{\varphi}$ denotes $\frac1{\mu (X)}\int_X\varphi\,d\mu$.
These operators depend on the dyadic grid we chose. We shall need the following technical lemma.
\begin{lemma}
$$
(\pi( f), g) = \ave{f} (T\chi_X, g-\ave{g}\chi_X) + \sli(\pi h_Q^j, h_R^i)(f, h_Q^j)(g, h_R^i),
$$
$$
(\pi_* (f), g) = \ave{g}(T^*\chi_X, f-\ave{f}\chi_X) + \sli (\pi_* h_Q^j, h_R^i)(f, h_Q^j)(g, h_R^i).
$$
\end{lemma}
\begin{proof}
The second equality follows from the first one and the definition of $\pi_*$. We prove the first equality. We will not write superscripts $i$ and $j$ in Haar functions.

We write
$$
\pi (f) = \ave{f}\pi(\chi_X) + \sli (f, h_Q^i)\pi(h_Q^i).
$$
Notice that
$$
\pi(\chi_X) = \sli(T\chi_X, h_Q^i)h_Q^i = T\chi_X - \ave{T\chi_X},
$$
and that $\pi(f)$ is orthogonal to $\chi_X$. Thus,
\begin{multline*}
(\pi (f), g) = (\pi (f), g-\ave{g}\chi_X) = \ave{f}(\pi(\chi_X), \sli (g, h_R^j)h_R^j) + \sli (\pi h_Q^i, h_R^j)(f, h_Q^i)(g, h_R^j) = \\ =\ave{f}(T\chi_X, g-\ave{g}\chi_X) + \sli (\pi h_Q^i, h_R^j)(f, h_Q^i)(g, h_R^j),
\end{multline*}
as desired.
The last equality is true because $\ave{T\chi_X}$ is orthogonal to $g-\ave{g}\chi_X$.
\end{proof}

Notice that $\pi, \pi^*$ depend on the random dyadic grid. We introduce a random operator
$$
\tilde{T}=Tf - \pi(f)-\pi_*(f).
$$
Now we state the following very useful lemma.
\begin{lemma}
$$
(Tf, g) = \pi_{good}^{-1}\E \sli_{\stackrel{Q, R}{\mbox{smaller is good}}}(\tilde{T}h_Q^i, h_R^j)(f,h_Q^i)(g, h_R^j) + \E (\pi ( f), g) + \E (\pi_* (f), g) + \ave{f}\ave{g}(T\chi_X, \chi_X).
$$
\end{lemma}
\begin{proof}
First, we write
$$
(Tf, g) = \sli (Th_Q^i, h_R^j)(f, h_Q^i)(g, h_R^j) + \ave{f}(T\chi_X, g) + \ave{g}(T^*\chi_X, f-\ave{f}\chi_{X}).
$$
We take expectations now. Notice that only the first term in the right-hand side depends on a dyadic grid. Therefore,
$$
(Tf, g)=\E \sli (Th_Q^i, h_R^j)(f, h_Q^i)(g, h_R^j) + \ave{f}(T\chi_X, g) + \ave{g}(T^*\chi_X, f-\ave{f}\chi_X).
$$
We focus on the first term. By the Theorem \ref{rand1}, we know that
\begin{multline}
\E \sli (Th_Q^i, h_R^j)(f, h_Q^i)(g, h_R^j) = \pi_{good}^{-1}\E \sli_{\mbox{smaller is good}} (Th_Q^i, h_R^j)(f, h_Q^i)(g, h_R^j) = \\ =\pi_{good}^{-1}\E \sli_{\mbox{smaller is good}} (\tilde{T}h_Q^i, h_R^j)(f, h_Q^i)(g, h_R^j) + \\+\pi_{good}^{-1}\E \sli_{\mbox{smaller is good}} (\pi h_Q^i, h_R^j)(f, h_Q^i)(g, h_R^j)+\pi_{good}^{-1}\E \sli_{\mbox{smaller is good}} (\pi_* h_Q^i, h_R^j)(f, h_Q^i)(g, h_R^j).
\end{multline}

The first term is one of those that we want to get in the right-hand side.

On the other hand, we want to get a result for paraproducts, similar to the Theorem \ref{rand1}. Indeed, it is clear that
$$
(\pi h_Q^i, h_R^j) = \av{h_Q^i}{R} (T\chi_X, h_R^j),
$$
which is non-zero only if $R\subset Q$, and $R\not = Q$. So,
\begin{multline}
\E \sli_{\mbox{smaller is good}} (\pi h_Q^i, h_R^j)(f, h_Q^i)(g, h_R^j) = \E \sli_{R\subset Q} \av{h_Q^i}{R}(T\chi_X, h_R^j)(f, h_Q^i)(g, h_R^j){\bf 1}_{R \; \mbox{is good}} = \\
= \E \sli_{R} (T\chi_X, h_R^j)(g, h_R^j){\bf 1}_{R \; \mbox{is good}} \sli_{Q\,: R\subsetneq Q} (f, h_Q^i)\av{h_Q^i}{R}.
\end{multline}
We now see that since $f=\ave{f}\chi_X+\sli_Q (f, h_Q^i)h_Q^i$, we have
$$
\av{f}{R}-\ave{f} = (f, \mu(R)^{-1}\chi_R) - \ave{f} =\sli_{Q\,: R\subsetneq Q} (f, h_Q^i)\av{h_Q^i}{R}= \sli_Q (f, h_Q^i) \av{h_Q^i}{R}.
$$
Therefore,
\begin{multline}
\E \sli_{R} (T\chi_X, h_R^j)(g, h_R^j){\bf 1}_{R \; \mbox{is good}} \sli_Q (f, h_Q^i)\av{h_Q^i}{R} = \E \sli_R (T\chi_X, h_R^j)(g, h_R^j){\bf 1}_{R \; \mbox{is good}} (\av{f}{R}-\ave{f}).
\end{multline}
Now it is clear that we can take the expectation inside (we have no $Q$ anymore, which was preventing us from doing that), and so we get
$$
\E \sli_{\mbox{smaller is good}} (\pi h_Q^i, h_R^j)(f, h_Q^i)(g, h_R^j) = \pi_{good}\E \sli_R (T\chi_X, h_R^j)(g, h_R^j)(\av{f}{R}-\ave{f}).
$$
Making all above steps backwards, we get
$$
\E \sli_{\mbox{smaller is good}} (\pi h_Q^i, h_R^j)(f, h_Q^i)(g, h_R^j) = \pi_{good}\E \sli (\pi h_Q^i, h_R^j)(f, h_Q^j)(g, h_R^j)
$$
 Therefore,
\begin{multline}
\pi_{good}^{-1}\E \sli_{\mbox{smaller is good}} (\pi h_Q^i, h_R^j)(f, h_Q^i)(g, h_R^j)+\pi_{good}^{-1}\E \sli_{\mbox{smaller is good}} (\pi_* h_Q^i, h_R^j)(f, h_Q^i)(g, h_R^j) = \\= \E \sli (\pi h_Q^i, h_R^j)(f, h_Q^i)(g, h_R^j)+ \E \sli (\pi_* h_Q^i, h_R^j)(f, h_Q^i)(g, h_R^j)=\\
= \E (\pi (f), g) + \E(\pi_* (f), g) - \E[ \ave{f}(T\chi_X, g-\ave{g}\chi_X)] - \E[ \ave{g}(T^*\chi_X, f-\ave{f}\chi_X)].
\end{multline}
We now use that last two terms do not depend on the dyadic grid, and so we drop expectations. Finally,
\begin{multline}
(Tf, g) = \E \sli_{\text{smaller is good}} (\tilde{T}h_Q^i, h_R^j)(f, h_Q^i)(g, h_R^j) + \E (\pi (f), g) + \E (\pi_* (f), g)  -\\- \ave{f}(T\chi_X, g-\ave{g}\chi_X) - \ave{g}(T^* \chi_X, f-\ave{f}\chi_X)+\ave{f}(T\chi_X, g) + \ave{g}(T^*\chi_X, f-\ave{f}\chi_X) = \\ = \E \sli_{\text{smaller is good}} (\tilde{T}h_Q^i, h_R^j)(f, h_Q^i)(g, h_R^j) + \E (\pi (f), g) + \E (\pi_* (f), g) + \ave{f}\ave{g}(T\chi_X, \chi_X).
\end{multline}
This is what we want to prove.
\end{proof}
The following lemma, which will be proved later, takes care of paraproducts.
\begin{lemma}\label{parapr}
The operators $\pi$, $\pi_*$ are bounded on $L^2(X, wd\mu)$, and
$$
\|\pi\|_{2, w}\leqslant C \cdot [w]_{2, \mu}.
$$
The same is true for $\pi_*$.
\end{lemma}
We postpone the proof of this lemma.
We also notice that the operator
$$
o(f)=\ave{f}\ave{T\chi_X}\chi_X
$$
is clearly bounded with desired constant. In fact, as $T$ is bounded in the unweighted $L^2$, we have $\ave{T\chi_X}^2\le \|T\|^2_{L^2}=:C_0$
$$
\|o(f)\|^2_{2, w} = \ave{f}^2\ave{T\chi_X}^2 w(X) \leqslant C_0 \ave{f^2 w} \ave{w^{-1}} w(X) \leqslant C_0 [w]_2 \|f\|^2_{2, w}.
$$
We, therefore, should take care only of the first term, with $\tilde{T}$. We now erase the tilde, and write $T$ instead of $\tilde{T}$. Even though $T$ is not a Calderon-Zygmund operator anymore, all further estimates are true for $T$ (i.e., for a CZO minus paraproducts), see, for example, \cite{HM} or \cite{HPTV}.

\subsection{Estimates of $\sigma_1$}
Our next step is to decompose $\sigma_1$ into random dyadic shifts. 
We write
\begin{multline}
\overline{\sigma_1(T)}=\sli_{\substack{\ell(Q)\geqslant \ell(R)\\ R \; \mbox{is good}}}(Th^j_Q, h^i_R)(f, h^j_Q)(g, h^i_R) = \\=
\E\sli_{\substack{\ell(Q)\geqslant\delta^{-r_0}\ell(R),\\R\subset Q, \\R\;\mbox{is good}}}(Th^j_{Q}, h^i_R)(f, h^j_Q)(g,h^i_R) + \\+
\E\sli_{\substack{\ell(R)\leqslant\ell(Q)<\delta^{-r_0}\ell(R),\\ R\subset Q,\\R\;\mbox{is good}}}(Th^j_{Q}, h^i_R)(f, h^j_Q)(g,h^i_R)+\\+
\E\sli_{\substack{\ell(R)\leqslant\ell(Q),\\ R\cap Q=\emptyset, \\R\;\mbox{is good}}}(Th^j_{Q}, h^i_R)(f, h^j_Q)(g,h^i_R).
\end{multline}

Essentially, we will prove that the norm of every expectation is bounded by 
$$
C(T)\cdot \E \sli_n \delta^{-\ep(T)\cdot n} \|\mathbb{S}_n \|.
$$
First, we state our choice for $\gamma$, which we have seen in the definition of good cubes.
\begin{defin}
Put
$$
\gamma=\frac{\ep}{2\cdot(\ep + \log_2(C))},
$$
where $C$ is the doubling constant of the function $\lambda$.
\end{defin}
\begin{zamech}
We remark that this choice of $\gamma$ make Lemmata \ref{in} and \ref{out} true. 
\end{zamech}

The estimate of the second sum is easy. In fact,
$$
\E\sli_{\substack{\ell(R)\leqslant\ell(Q)<\delta^{-r_0}\ell(R),\\ R\subset Q, \\R\;\mbox{is good}}}(Th^j_{Q}, h^i_R)(f, h^j_Q)(g,h^i_R) \leqslant Cr_0\,[w]_2 \|f\|\|g\|.
$$
This is bounded by at most $r_0$ expressions for shifts of bounded complexity, so just see \cite{NV}. For more details, see \cite{HPTV}
%%%%%%%%
%\begin{zamech}
%Napishete chut' podrobnee?
%\end{zamech}
%%%%%%%%%%%%

We denote
\begin{align*}
&\Sigma_{in}=\E\sli_{\substack{\ell(Q)\geqslant\delta^{-r_0}\ell(R),\\R\subset Q, \\R\;\mbox{is good}}}(Th^j_{Q}, h^i_R)(f, h^j_Q)(g,h^i_R),\\
&\Sigma_{out}=\E\sli_{\substack{\ell(R)\leqslant\ell(Q),\\ R\cap Q=\emptyset, \\R\;\mbox{is good}}}(Th^j_{Q}, h^i_R)(f, h^j_Q)(g,h^i_R).
\end{align*}

\subsection{Estimate of $\Sigma_{in}$.}
We use the following lemma.
\begin{lemma}
\label{in}
Let $T$ be as before; suppose $\ell(Q)\geqslant \delta^{-r_0}\ell(R)$ and $R\subset Q$. Let $Q_1$ be the son of $Q$ that contains $R$. Then
$$
|(Th^j_Q, h^i_R)|\lesssim \frac{\ell(R)^{\frac{\ep}{2}}}{\ell(Q)^{\frac{\ep}{2}}} \left(\frac{\mu(R)}{\mu(Q_1)}\right)^{\frac{1}{2}}.
$$
\end{lemma}
We notice that $\mu(Q_1)\asymp \mu(Q)$.

%%%%%%%%%%%%%%
%\begin{zamech}
%Chego-to ya perestal eto ponimat', hotya sam zhe i napisal.
%
%Esli $\mu$ doubling, to no problem, a esli s $\lambda$??? Togda $\mu(Q)\leqslant \lambda(x, \ell(Q)) \leqslant C \lambda(x, \ell(Q_1))$, no poslednee ne ocenit' cherez $\mu(Q_1)$!!!
%
%Vot tut-to nado ispol'zovat', chto $\mu$ --- doubling.
%\end{zamech}
%%%%%%%%%%%%%%%%%%

We write
$$
\Sigma_{in} = \sli_{n\geqslant r_0}\sli_{\ell(Q)=\delta^{-n}\ell(R), R\;\mbox{is good}, R\subset Q}(Th^j_{Q}, h^i_R)(f, h^j_Q)(g,h^i_R),
$$
\begin{multline}
|\Sigma_{in} |\leqslant  \sli_{n\geqslant r_0}\sli_{\substack{\ell(Q)=\delta^{-n}\ell(R), \\R\;\mbox{is good}, \\R\subset Q}}|(Th^j_{Q}, h^i_R)||(f, h^j_Q)||(g,h^i_R)|\leqslant \\ \leqslant C\sli_{n\geqslant r_0}\sli_{\substack{\ell(Q)=\delta^{-n}\ell(R),\\ R\;\mbox{is good}, \\R\subset Q}}\frac{\ell(R)^{\frac{\ep}{2}}}{\ell(Q)^{\frac{\ep}{2}}} \left(\frac{\mu(R)}{\mu(Q)}\right)^{\frac{1}{2}}|(f, h^j_Q)||(g,h^i_R)| =\\=
C\sli_{n\geqslant r_0}\delta^{\frac{n\ep}{2}}\sli_{\substack{\ell(Q)=\delta^{-n}\ell(R),\\ R\;\mbox{is good}, \\R\subset Q}} \left(\frac{\mu(R)}{\mu(Q)}\right)^{\frac{1}{2}}|(f, h^j_Q)||(g,h^i_R)|.
\end{multline}
We {\bf fix} functions $f$ and $g$ and define $S_n$ as an operator with the following quadratic form:
$$
(S_n u,v) = \sli_{\substack{\ell(Q)=\delta^{-n}\ell(R), \\R\;\mbox{is good}, \\R\subset Q}} \pm \left(\frac{\mu(R)}{\mu(Q)}\right)^{\frac{1}{2}}(u, h^j_Q)(v, h^i_R),
$$
where $\pm$ is chosen so $|(f, h^j_Q)||(g, h^i_R)| = \pm (f, h^j_Q)(g, h^i_R)$.
Then clearly $S_n$ is a dyadic shift of complexity $n$, and so, see Section \ref{shiftest},

$$
|(S_n f, g)|\leqslant Cn^a [w]_2\|f\|_w \|g\|_{w^{-1}}.
$$
%So, using the main result of \cite{NV} or \cite{Tr}, we obtain
Therefore,
$$
|\Sigma_{in}|\leqslant \sli_{n}C n^a \delta^{\frac{n\ep}{2}}[w]_2\|f\|_w \|g\|_{w^{-1}} \leqslant C[w]_2\|f\|_w \|g\|_{w^{-1}}.
$$
\subsection{Estimates for $\Sigma_{out}$}
We use the following lemma from ~\cite{HM}.
\begin{lemma}\label{out}
Let $T$ be as before, $\ell(R)\leqslant \ell(Q)$ and $R\cap Q=\emptyset$. Then the following holds
$$
|(Th^j_Q, h^i_R)| \lesssim \frac{\ell(Q)^{\frac{\ep}{2}}\ell(R)^{\frac{\ep}{2}}}{D(Q, R)^{\ep} \sup_{z\in R}\lambda(z, D(Q,R))} \mu(Q)^{\frac{1}{2}}\mu(R)^{\frac{1}{2}},
$$
where $D(Q, R) = \ell(Q)+\ell(R)+\textup{dist}(Q, R)$.
\end{lemma}
\begin{zamech}
We should clarify one thing here. If $T$ was a Calderon-Zygmund operator, this estimate would be standard, see \cite{NTV}, \cite{NTV2} or, for metric spaces, \cite{HM}. We, however, subtracted from $T$ two operators: paraproduct and adjoint to paraproduct. However, an easy argument (see \cite{HPTV}) shows that if $R\cap Q=\emptyset$, then $(Th_Q^j, h_R^i)=(\tilde{T}h_Q^j, h_Q^i)$ (for the definition of $\tilde{T}$ see Lemma \ref{parapr} and thereon).
\end{zamech}
Suppose now that $D(Q,R)\sim \delta^{-s}\ell(Q)$. We ask the question: what is the probability
$$
\mathbb{P}(R\subset Q^{(s+s_0+10)} | Q,R\in D_{\omega}),
$$
where $s_0$ is a sufficiently big number.
We use the Lemma \ref{sloy}. Suppose that $R\cap Q^{(s+s_0+10)}=\emptyset$. Suppose also $R=R_x$ (so $x$ is the ``center'' of $R$). Then
\begin{multline}
dist(x, Q^{(s+s_0+10)})\leqslant dist(x, Q)\leqslant dist(Q,R) \leqslant C\delta^{-s}\ell(Q) = C\delta^{-s}\delta^{s+s_0+10}\ell(Q^{(s+s_0+10)}) =\\= C\delta^{s_0 + 10} \ell(Q^{(s+s_0+10)}).
\end{multline}
So $x\in \delta_{Q^{(s+s_0+10)}}(\delta^{s_0+10}))$, and the probability of this is estimated by $\delta^{\eta(s_0+10)}<\frac{1}{2}$ for sufficiently big $s_0$ (we remind that $\eta = \log_{\delta}(1-a))$. Therefore,
$$
\mathbb{P}(R\subset Q^{(s+s_0+10)} | Q,R\in D_{\omega})\geqslant \frac{1}{2}.
$$
So
%$$
%\Sigma_e \leqslant 2\E\sli_{\ell(Q)\geqslant\delta^{-r_0}\ell(R), R\cap Q=\emptyset}(T%h^j_{Q}, h^i_R)(f, h^j_Q)(g,h^i_R)\mathbf{1}_{R \,\mbox{is good}}\mathbf{1}_{R\subset Q^{(s+s_0+10)}}.
%$$
\begin{multline}
|\Sigma_{out}| \leqslant 2\E\sli_{t,s}\sli_{\substack{\ell(Q)=\delta^{-t}\ell(R), \\D(Q,R)\sim \delta^{-s}\ell(Q),\\R\cap Q=\emptyset}}|(Th^j_{Q}, h^i_R)||(f, h^j_Q)||(g,h^i_R)|\mathbf{1}_{R \,\mbox{is good}}\mathbf{1}_{R\subset Q^{(s+s_0+10)}}\leqslant \\
\leqslant 2\E\sli_{t,s}\sli_{\substack{\ell(Q)=\delta^{-t}\ell(R), \\D(Q,R)\sim \delta^{-s}\ell(Q),\\R\cap Q=\emptyset \\ R,Q\subset Q^{s+s_0+10}}}\frac{\ell(Q)^{\frac{\ep}{2}}\ell(R)^{\frac{\ep}{2}}}{D(Q, R)^{\ep} \sup_{z\in R}\lambda(z, D(Q,R))} \mu(Q)^{\frac{1}{2}}\mu(R)^{\frac{1}{2}}|(f, h^j_Q)||(g,h^i_R)|\mathbf{1}_{R \,\mbox{is good}} \leqslant \\ \leqslant
2\E\sli_{t,s}\sli_{\substack{\ell(Q)=\delta^{-t}\ell(R), \\D(Q,R)\sim \delta^{-s}\ell(Q),\\R\cap Q=\emptyset,\\ R, Q\subset Q^{s+s_0+10}}}\delta^{\frac{t\ep}{2}}  \left(\frac{\ell(Q)}{D(Q, R)}\right)^{\ep} \frac{\mu(Q)^{\frac{1}{2}}\mu(R)^{\frac{1}{2}}}{\sup_{z\in R}\lambda(z, D(Q,R))}|(f, h^j_Q)||(g,h^i_R)|\mathbf{1}_{R \,\mbox{is good}} \leqslant \\ \leqslant
C2\E\sli_{t,s}\delta^{\frac{t\ep}{2}} \delta^{s\ep}\sli_{\substack{\ell(Q)=\delta^{-t}\ell(R), \\D(Q,R)\sim \delta^{-s}\ell(Q),\\R\cap Q=\emptyset,\\ R, Q\subset Q^{s+s_0+10}}}  \frac{\mu(Q)^{\frac{1}{2}}\mu(R)^{\frac{1}{2}}}{\sup_{z\in R}\lambda(z, D(Q,R))}|(f, h^j_Q)||(g,h^i_R)|\mathbf{1}_{R \,\mbox{is good}}.
\end{multline}
We now define $S_n$ as we did before:
$$
(S_n u, v)=\sli_{\substack{\ell(Q)=\delta^{-t}\ell(R), \\D(Q,R)\sim \delta^{-s}\ell(Q),\\R\cap Q=\emptyset,\\ R, Q\subset Q^{s+s_0+10}}}  \pm\frac{\mu(Q)^{\frac{1}{2}}\mu(R)^{\frac{1}{2}}}{\sup_{z\in R}\lambda(z, D(Q,R))}(u, h^j_Q)(v,h^i_R)\mathbf{1}_{R \,\mbox{is good}}.
$$
We need to estimate the coefficient. We write
\begin{multline}
\lambda(z, D(Q,R))\sim \lambda(z, \delta^{-s}\ell(Q)) \sim \lambda(z, \delta^{-s-s_0-10}\ell(Q)) \sim \\ \sim \lambda(z, \ell(Q^{(s+s_0+10)})) \sim \lambda(z, diam(Q^{(s+s_0+10)}))  \geqslant \mu(B(z, diam(Q^{(s+s_0+20)}))) \geqslant \\ \geqslant \mu(Q^{(s+s_0+10)}),
\end{multline}
and therefore
$$
|\pm\frac{\mu(Q)^{\frac{1}{2}}\mu(R)^{\frac{1}{2}}}{\sup_{z\in R}\lambda(z, D(Q,R))}| \leqslant C \frac{\mu(Q)^{\frac{1}{2}}\mu(R)^{\frac{1}{2}}}{\mu(Q^{s+s_0+10})}.
$$
We notice that $C$ does not depend on $s$ since we used the doubling property of $\lambda$ only for transmission from $\delta^{-s}\ell(Q)$ to $\delta^{-s-s_0-10}\ell(Q)$.

We conclude that $S_n$ is a dyadic shift of complexity at most $C(s+t)$. Therefore, see Section \ref{shiftest},
$$
|\Sigma_{out}|\leqslant 2C\E\sli_{t,s}\delta^{\frac{t\ep}{2}} \delta^{s\ep} (s+t)^a [w]_2\|f\|_w \|g\|_{w^{-1}} \leqslant C[w]_2\|f\|_w \|g\|_{w^{-1}},
$$
and our proof is completed.

\section{Paraproducts and Bellman function}
\label{para}
Now we will prove the Lemma \ref{parapr}.

We remind that the quadratic form of our paraproduct $\pi$ is the following:
$$
(\pi(f),g):=\sli_R\sli_i \langle f\rangle_{\mu, R}  (T\chi_X, h^i_R) (g, h^i_R)\,.
$$

Operator $T$ is bounded in $L^2(\mu)$ and $\mu$ is doubling. Therefore, it is well known that
coefficients $b_R:=b^i_R:=(T\chi_X, h^i_R)$ satisfy Carleson condition for any of our lattices of Christ's dyadic cubes:

\begin{equation}
\label{CarlQR}
\forall  Q\in \mathcal{D} \,\, \sli_{R\in \mathcal{D},\, R\subset Q} |b_R|^2  \le B\, \mu(Q)\,.
\end{equation}
The best constant $B$ here is called the Carleson constant and it is denoted by $\|b\|_{C}$. It is known that for our $b_R:=(T\chi_X, h^i_R)$ Carleson constant is bounded by $B_T:= C\, \|T\|_{L^2(\mu)\rightarrow L^2(\mu)}$.

If we would be on the line with Lebesgue measure $\mu$ and $w$ would be a usual weight in $A_2$, then the sum would follow the estimate of O. Beznosova \cite{Bez}:
\begin{equation}
\label{bez}
|\pi_{T\chi_X}(f,g)| \le C\, \sqrt{B_T}\|w\|_{A_2}\,.
\end{equation}

But the same is true in our situation. To prove that, one should analyze the proof in \cite{Bez} and see that it used always conditions on $w$ and $b$ separately. They were always split by Cauchy--Schwarz inequality.  The only inequality, where $w$ and $b$ meet was of the type: let $Q$ be a Christ's cube of a certain lattice, then

\begin{equation}
\label{ainfty}
\sli_{R\subset Q, \,R\in\mathcal{D}} \langle w\rangle_{\mu, R} b_R^2 \le [w]_{A_\infty} \|b\|_{C} \int_Q w\,d\mu\, ,
\end{equation}
where
$$
[w]_{A_\infty} = \sup \frac{1}{\mu(B)}\ili_{B} wd\mu \cdot \exp\left(-\frac{1}{\mu(B)}\ili_{B}wd\mu\right).
$$

Let us explain the last inequality. We write
$$
\langle w \rangle_{\mu, R} \leqslant [w]_{A_\infty} \cdot \exp\left(\langle w \rangle_{\mu, R}\right) = [w]_{A_\infty} \cdot \exp\left(2\langle w^{\frac{1}{2}} \rangle_{\mu, R}\right) \leqslant [w]_{A_\infty} \langle w^{\frac{1}{2}}\rangle_{\mu, R}^2 \leqslant [w]_{A_\infty} \inf_{x\in R} M(w^{\frac{1}{2}}\chi_R)^2.
$$

Finally, we notice that $\{b_R^2\}$ is a Carleson sequence, and finish our explanation with the following well known theorem.
\begin{theorem}
Suppose $\{\alpha_K\}$ is a Carleson sequence. Then for any positive function $F$ the following inequality holds:
$$
\sum_{K}\alpha_K \inf_K F(x) \leqslant \int F(x)d\mu(x).
$$
\end{theorem}

%In \cite{Bez} this is done by Bellman function technique. In fact, this follows from the fact that integrating over Carleson measure is the same as integrating maximal function. Using now the fact that $w\in A_{\infty}$ (of course, as it is in $A_2$), we can write $ \langle w\rangle_{\mu, R}\le \|w\|_{A_{\infty}} \langle w^{1/2}\rangle_{\mu, R}^2$, so we integrate $\int( M_{\mu}1_Q\,w^{1/2})^2\,d\mu$. Independently of anything this is known to be bounded by $C\,\int_Q w\, d\mu$.

In all other estimates in \cite{Bez} the sums with $\Delta_Q w$ (see the  definition before Lemma 3.2 of \cite{NV}) and the sums with $b$ are always estimated separately. The sums where the terms contain the product of $\Delta_Q w$ and $b_Q$ never got estimated by Bellman technique: they got split first. Then \eqref{bez} follows in our metric situation as well.

%We refer again to ~\cite{Tr} for another proof of estimate for paraproduct.
\section{Weighted estimates for dyadic shifts via Bellman function}
\label{shiftest}

This section is here just for the sake of completeness. In fact, it just repeats the article of Nazarov--Volberg \cite{NV}. In this section we prove the following theorem.
\begin{theorem}
Let $\mathbb{S}_{m,n}$ be a dyadic shift of complexity $(n,m)$. Then
$$
\|\mathbb{S}_{m,n}\|_{wd\mu}\leqslant C (m+n+1)^a [w]_{2, \mu}.
$$
\end{theorem}

\begin{zamech}
We notice that the best known $a$ is equal to one. It can be gotten using the technique from \cite{HLM+} or from \cite{Tr}. However, for the application we made in the previous sections, namely, the linear $A_2$ bound for an arbitrary Calder\'on--Zygmund operator on geometrically doubling metric space, the actual value of $a$ is not important.
\end{zamech}

We now give formal definitions. Let $h_Q^i$ be Haar functions as before, normalized in $L^2$. We also denote by $g(Q)$ the generation of a ``dyadic cube'' $Q$. Then by $\mathbb{S}_{m,n}$ we denote an operator
$$
f\rightarrow \sum_{L\in \cD} \int_L a_L(x,y)f(y)dy\,,
$$
where
$$
a_L(x,y) =\sum_{\stackrel{ I\subset L, J\subset L}{g(I)= g(L)+m, \, g(J)= g(L)+n}}c_{L,I,J} h_J^j(x)h_I^i(y)\,.
$$

We denote $\sigma=w^{-1}$. We begin with the following  lemma.
\begin{lemma}
\label{decomp}
$$
h_I^j = \al_I ^jh_I^{w,j} + \beta_I^j \chi_I\,,
$$
where

1)  $|\alpha_I^j| \le \sqrt{\langle w\rangle_{\mu, I}}$,

2)$ |\beta_I^j| \le \frac{(h_I^{ w,j}, w)_{\mu}}{w(I)}$, where $w(I):= \int_I w\,d\mu$,

3) $\{h_I^{w,j}\}_{I} $ is supported on $I$, orthogonal to constants in $L^2(w\,d\mu)$,

4) $h_I^{w,j}$ assumes on each son $s(I)$ a constant value,

5) $\|h_I^{w,j}\|_{L^2(w\,d\mu)}=1$.
\end{lemma}

\bigskip

\noindent{\bf Definition.}
Let
$$
\Delta_I w:= \sum_{\text{sons of}\,\,I} |\langle w\rangle_{\mu, s(I)}-\langle w\rangle_{\mu, I}|\,.
$$

It is a easy to see that the doubling property of measure $\mu$ implies
\begin{equation}
\label{delta}
|(h_I^{ w,j}, w)_{\mu}|\le C\, (\Delta_I w)\, \mu(I)^{1/2}\,.
\end{equation}

Therefore, the property 2) above can be rewritten as

\medskip

2') $|\beta_I^j|\le C\,\frac{|\Delta_I w|}{\langle w\rangle_{\mu, I}} \frac{1}{\mu (I)^{1/2}}$.

\vspace{.2in}

Fix $\phi\in L^2(w\,d\mu), \psi\in L^2(\sigma)$. We need to prove
\begin{equation}
\label{main}
|(\sha_{m,n}\phi w,\psi\sigma)|\le C\, (n+m+1)^a\|\phi\|_{w}\|\psi\|_{\sigma}\,.
\end{equation}

We  estimate $(\sha_{m,n}\phi w, \psi \sigma)$ as

$$
|\sum_L\sum_{ I, J} c_{L,I,J}(\phi w, h_I)_{\mu}(\psi \sigma, h_J)_{\mu}|\le
$$
$$
\sum_L\sum_{ I,J} |c_{L,I,J} (\phi w, h^w_I)_{\mu}\sqrt{\langle w \rangle_{\mu, I}}(\psi \sigma,h^{\sigma}_J)_{\mu}|\sqrt{\langle \sigma \rangle_{\mu, J}}|\,+
$$
$$
\sum_L\sum_{I,J} |c_{L,I,J} \langle \phi w\rangle_{\mu, I}\frac{\Delta_I w}{\langle w \rangle_{\mu, I}}(\psi \sigma,h^{\sigma}_J)_{\mu}\sqrt{\langle \sigma \rangle_{\mu, J}}\sqrt{I}|\,+
$$
$$
\sum_L\sum_{I,J} |c_{L,I,J} \langle \psi \sigma\rangle_{\mu, J}\frac{\Delta_J\sigma}{\langle \sigma \rangle_{\mu, J}}(\phi w,h^{w}_I)_{\mu}\sqrt{\langle w \rangle_{\mu, I}}\sqrt{J}|\,+
$$
$$
\sum_L\sum_{I,J}|c_{L,I,J} \langle \phi w\rangle_{\mu, I}\langle \psi \sigma\rangle_{\mu, J} \frac{\Delta_I w}{\langle w \rangle_{\mu, I}} \frac{\Delta_J\sigma}{\langle \sigma \rangle_{\mu, J}}\sqrt{I}\sqrt{J}|=: I + II +III +IV\,.
$$

We can notice  that  because $|c_{L,I,J}|\le \frac{\sqrt{\mu(I)}\sqrt{\mu(J)}}{\mu(L)}$  each sum inside $L$ can be estimated  by a perfect product of $S$ and $R$ terms, where
$$
R_L(\phi w):= \sum_{I\subset L...} \langle \phi w\rangle_{\mu, I}  \frac{|\Delta_I w|}{\langle w \rangle_{\mu, I}}\frac{\mu(I)}{\sqrt{\mu(L)}}
$$
$$
S_L(\phi w) := \sum_{I\subset L...} (\phi w, h^w_I)_{\mu}\sqrt{\langle w \rangle_{\mu, I}}\frac{\sqrt{\mu(I)}}{\sqrt{\mu(L)}}
$$
and the corresponding terms for $\psi \sigma$.
So we have
$$
I \le \sum_L S_L(\phi w)S_L(\psi\sigma)
,\,
II \le \sum_L S_L(\phi w)R_L(\psi\sigma),\,\,
\,
$$
$$
III\le \sum_L R_L(\phi w)S_L(\psi\sigma)
,\,\,\,
IV\le \sum_L R_L(\phi w)R_L(\psi\sigma)\,.
$$
Now
\begin{equation}
\label{sl}
S_L(\phi w) \le \sqrt{\sum_{I\subset L...} |(\phi w, h^w_I)_{\mu}|^2}\sqrt{\La w\Ra_{\mu, L}}\,,\,\,\,S_L(\psi \sigma) \le \sqrt{\sum_{J\subset L...} |(\psi \sigma, h^\sigma_J)|^2}\sqrt{\La \sigma\Ra_{\mu, L}}
\end{equation}
Therefore,
\begin{equation}
\label{I}
I\le C\QQ^{1/2} \|\phi\|_w\|\psi\|_{\sigma}\,.
\end{equation}

Terms $II, III$ are symmetric, so consider $III$.
Using Bellman function $(xy)^{\al}$ one can prove now
\begin{lemma}
\label{uval}
The sequence
$$
\tau_I := \La w\Ra_{\mu, I}^{\al}\La \sigma\Ra_{\mu, I}^{\al}\bigg(\frac{|\Delta_I w|^2}{\La w \Ra_{\mu, I}^2} + \frac{|\Delta_I \sigma|^2}{\La \sigma \Ra_{\mu, I}^2}\bigg) \mu(I)
$$
form a Carleson measure with Carleson constant  at most $c_{\al}Q^{\al}$, where $Q:=[w]_{A_2}$ for any $\al\in (0, 1/2)$.
\end{lemma}
\begin{proof}
We need a very simple

\medskip

\noindent{\bf Sublemma}.
Let $Q> 1, 0<\alpha<\frac12$. In domain  $\Omega_Q:=\{(x,y): X>o, y>0, 1<xy\le Q$ function $B_Q(x,y):=x^{\al}y^{\al}$ satisfies the following estimate of its Hessian matrix  (of its second differential form, actually)
$$
-d^2 B_Q(x,y)\ge \al(1-2\al)x^{\al}y^{\al}\bigg(\frac{(dx)^2}{x^2} +\frac{(dy)^2}{y^2}\bigg)\,.
$$
The form $-d^2 B_Q(x,y)\ge 0$ everywhere in $x>0, y>0$. Also obviously $0\le B_Q(x,y) \le Q^{\al}$ in $\Omega_Q$.
\begin{proof}
Direct calculation.
\end{proof}

\medskip

Fix now a Christ's cube $I$ and let $s_i(I), i=1,...,M$, be all its sons. Let $a=(\La w\Ra_{\mu, I}, \La \sigma\Ra_{\mu, I})$, $b_i=(\La w\Ra_{\mu,s_i( I)}, \La \sigma\Ra_{\mu, s_i(I)})$, $i=1,\dots, M$, be points--obviously--in $\Omega_Q$, where $Q$ temporarily means $[w]_{A_2}$. Consider $c_i(t)=a(1-t)+ b_it, 0\le t \le 1$ and $q_i(t):= B_Q(c_i(t))$.  We want to use Taylor's formula
\begin{equation}
\label{Lagr}
q_i(0)-q_i(1) = -q'_i(0) - \int_0^1dx\int_0^x q_i''(t)\,dt\,.
\end{equation}
Notice two things: Sublemma  shows that $-q_i''(t) \ge 0$ always. Moreover, it shows that if $t\in [0,1/2]$,  then we have the following qualitative estimate holds
\begin{equation}
\label{wI}
-q_i''(t) \ge c\,( \La w\Ra_{\mu, I} \La \sigma\Ra_{\mu, I})^{\al}\bigg(\frac{(\La w\Ra_{\mu,s_i( I)}-\La w\Ra_{\mu, I})^2}{\La w\Ra_{\mu, I}^2} +\frac{(\La \sigma\Ra_{\mu,s_i( I)}-\La \sigma\Ra_{\mu, I})^2}{\La \sigma\Ra_{\mu, I}^2} \bigg)
\end{equation}
This requires a small explanation. If we are on the segment $[a, b_i]$, then the first coordinate of such a point cannot be larger than $C\, \La w\Ra_{\mu, I}$, where $C$ depends only on doubling of $\mu$ (not $w$). This is obvious. The same is true for the second coordinate with the obvious change of $w$ to $\sigma$. But there is no such type of estimate from below on this segment:  the first coordinate cannot be smaller than $k\, \La w\Ra_{\mu, I}$, but $k$ may (and will) depend on the doubling of $w$ (so ultimately on its $[w]_{A_2}$ norm. In fact, at the ``right" endpoint of $[a, b_i]$. The first coordinate is $\La w\Ra_{\mu, s_i(I)}\le \int_I\,w\,d\mu/ \mu(s_i(I)) \le C\,  \int_I\,w\,d\mu/ \mu(I))=C\, \La w\Ra_{\mu, I}$, with $C$ only depending on the doubling of $\mu$. But the estimate from below will involve the doubling of $w$, which we must avoid. But if $t\in [0,1/2]$, and we are on the ``left half" of interval $[a, b_i]$ then obviously the first coordinate is $\ge \frac12 \La w\Ra_{\mu, I}$ and the second coordinate is $\ge \frac12 \La \sigma\Ra_{\mu, I}$.

We do not need to integrate $-q_i''(t)$ for all $t\in [0,1]$ in \eqref{Lagr}. We can only use integration over $[0,1/2]$  noticing that $-q_i''(t)\ge 0$ otherwise. Then the chain rule
$$
q_i''(t)=(B_Q(c_i(t))''=(d^2B_Q(c_i(t) (b_i-a), b_i-a)
$$
immediately gives us \eqref{wI} with constant $c$ depending on the doubling of $\mu$ but {\it independent} of the doubling of $w$.

\medskip

Next step is to add all \eqref{Lagr}, with convex coefficients $\frac{\mu(s_i(I))}{\mu(I)}$, and to notice that $\sum_{i=1}^M\frac{\mu(s_i(I))}{\mu(I)} q_i'(0) =\nabla B_{Q}(a) \sum_{i=1}^M\cdot (a-b_i)\frac{\mu(s_i(I))}{\mu(I)}=0$, because by definition
$$
a= \sum_{i=1}^M \frac{\mu(s_i(I))}{\mu(I)}\,b_i\,.
$$
Notice that the addition of all \eqref{Lagr}, with convex coefficients $\frac{\mu(s_i(I))}{\mu(I)}$ gives us now ( we take into account \eqref{wI} and positivity of $-q_i''(t)$)
$$
B_Q(a)- \sum_{i=1}^M \frac{\mu(s_i(I))}{\mu(I)}\,B_Q(b_i) \ge c\,c_1\,( \La w\Ra_{\mu, I} \La \sigma\Ra_{\mu, I})^{\al}\sum_{i=1}^M\bigg(\frac{(\La w\Ra_{\mu,s_i( I)}-\La w\Ra_{\mu, I})^2}{\La w\Ra_{\mu, I}^2} +\frac{(\La \sigma\Ra_{\mu,s_i( I)}-\La \sigma\Ra_{\mu, I})^2}{\La \sigma\Ra_{\mu, I}^2} \bigg)\,.
$$
We used here the doubling of $\mu$ again, by noticing that $\frac{\mu(s_i(I))}{\mu(I)}\ge c_1$ (recall that $s_i(I)$ and $I$ are almost balls of comparable radii).
We rewrite the previous inequality using our definition of $\Delta_I w, \Delta_I\sigma$ listed above as follows
$$
\mu(I)\,B_Q(a)- \sum_{i=1}^M \mu(s_i(I))\,B_Q(b_i) \ge c\,c_1\,( \La w\Ra_{\mu, I} \La \sigma\Ra_{\mu, I})^{\al}\bigg(\frac{(\Delta_I w)^2}{\La w\Ra_{\mu, I}^2} +\frac{(\Delta_I\sigma)^2}{\La \sigma\Ra_{\mu, I}^2} \bigg)\mu(I)\,.
$$
Notice that $B_Q(a)=\La w\Ra_{\mu, I}\La\sigma\Ra_{\mu,I}$. Now we iterate the above inequality and get for any of Christ's dyadic $I$'s:
$$
\sum_{J\subset I\,, J\in\cD} ( \La w\Ra_{\mu, J} \La \sigma\Ra_{\mu, J})^{\al}\bigg(\frac{(\Delta_J w)^2}{\La w\Ra_{\mu, J}^2} +\frac{(\Delta_J\sigma)^2}{\La \sigma\Ra_{\mu, J}^2} \bigg)\mu(J) \le C\, Q^{\al}\mu(I)\,.
$$
This is exactly the Carleson property of the measure $\{\tau_I\}$ indicated in our Lemma \ref{uval}, with Carleson constant $C\,Q^{\al}$. The proof showed that $C$ depended only on $\al\in (0, 1/2)$ and on the doubling constant of measure $\mu$.

\end{proof}

Now, using this lemma, we start to estimate our $S_L$'s and $R_L$'s. For $S_L(\psi\sigma)$ we already had  estimate \eqref{sl}.

To estimate $R_L(\phi w)$ let us denote by $\cP_L$  maximal stopping intervals $K\in \cD, K\subset L$, where the stopping criteria are 1) either $\frac{|\Delta_K w|}{\La w\Ra_{\mu, K}} \ge \frac{1}{m+n+1}$, or  $\frac{|\Delta_K \sigma|}{\La \sigma\Ra_{\mu, K}} \ge \frac{1}{m+n+1}$, or
2) $g(K)= g(L)+m$.

\begin{lemma}
\label{sbor}
If $K$ is any stopping interval then
\begin{equation}
\label{K}
\sum_{I\subset K, \ell(I)=2^{-m} \ell(L) } |\La \phi  w\Ra_{\mu, I} | \frac{|\Delta_I w|}{\La w\Ra_{\mu, I}}\frac{\mu(I)}{\sqrt{\mu(L)}} \le 2e^{\al} (m+n+1) \La |\phi | w\Ra_{\mu, K} \frac{\sqrt{\mu(K)}}{\sqrt{\mu(L)}}\sqrt{\tau_K}\La w\Ra_{\mu, L}^{-\al/2} \La \sigma\Ra_{\mu, L}^{-\al/2}\,.
\end{equation}
\end{lemma}

\begin{proof}
If we stop by the first criterion, then
\begin{multline*}
\sum_{I\subset K, \ell(I)=2^{-m} \ell(L) } |\La \phi w\Ra_{\mu, I} | \frac{|\Delta_I w|}{\La w\Ra_{\mu, I}}\frac{\mu(I)}{\sqrt{\mu(L)}} \le 2\sum_{I\subset K, \ell(I)=2^{-m} \ell(L) } |\La \phi w\Ra_{\mu, I} | \mu(I) \frac{1}{\mu(K)}\frac{\mu(K)}{\sqrt{\mu(L)}}\le \\ \le 2\,\La | \phi | w\Ra_{\mu, K}  \frac{\mu(K)}{\sqrt{\mu(L)}}
\le 2(m+n+1) \La |\phi | w\Ra_{\mu, K}\bigg( \frac{|\Delta_K w|}{\La w\Ra_{\mu, K}} + \frac{|\Delta_K \sigma|}{\La \sigma\Ra_{\mu, K}} \bigg)\frac{\mu(K)}{\sqrt{\mu(L)}}\le\\ \le  2(m+n+1) \La |\phi | w\Ra_{\mu, K} \frac{\sqrt{\mu(K)}}{\sqrt{\mu(L)}}\sqrt{\tau_K}\La w\Ra_{\mu, K}^{-\al/2} \La \sigma\Ra_{\mu, K}^{-\al/2}\,.
\end{multline*}
Now  replacing $\La w\Ra_{\mu, K}^{-\al/2} \La \sigma\Ra_{\mu, K}^{-\al/2}$ by $\La w\Ra_{\mu, L}^{-\al/2} \La \sigma\Ra_{\mu, L}^{-\al/2}$ does not grow the estimate by more than $e^{\al}$ as all pairs of son/father intervals  larger than $K$ and smaller than $L$ will have there averages compared by constant at most $1\pm \frac1{m+n+1}$. And there are at most $m$ such intervals between $K$ and $L$.

If we stop by the second criterion, then $K$ is one of $I$'s, $g(I)=g(L)+m$, and
$$
\ |\La \phi w\Ra_{\mu, I} |\frac{|\Delta_I w|}{\La w\Ra_{\mu, I}}\frac{\mu(I)}{\sqrt{\mu(L)}} \le | \La \phi w\Ra_{\mu, K} | \frac{\mu(K)}{\sqrt{\mu(L)}}\frac{|\Delta_K w|}{\La w\Ra_{\mu, K}}\le  \La |\phi | w\Ra_{\mu, K}  \frac{\sqrt{\mu(K)}}{\sqrt{\mu(L)}}\sqrt{\tau_K}\La w\Ra_{\mu, K}^{-\al/2} \La \sigma\Ra_{\mu, K}^{-\al/2}\,.
$$
Now  we replace $\La w\Ra_{\mu, K}^{-\al/2} \La \sigma\Ra_{\mu, K}^{-\al/2}$ by $\La w\Ra_{\mu, L}^{-\al/2} \La \sigma\Ra_{\mu, L}^{-\al/2}$  as before.

\end{proof}

Now
$$
R_L(\phi w) \le C(m+n+1) \La w\Ra_{\mu, L}^{-\al/2} \La \sigma\Ra_{\mu, L}^{-\al/2}\sum_{K\, \in \cP_L}  \La |\phi | w\Ra_{\mu, K} \frac{\sqrt{\mu(K)}}{\sqrt{\mu(L)}}\sqrt{\tau_K}
$$
$$
\le  C(m+n+1) \La w\Ra_{\mu, L}^{-\al/2} \La \sigma\Ra_{\mu, L}^{-\al/2} \bigg(\sum_{K\,  \in \cP_L}  \La |\phi | w\Ra_{\mu, K}^2 \frac{{\mu(K)}}{\mu(L)}\bigg)^{1/2} (\wt\tau_L)^{1/2}\,,
$$
where
$$
\wt\tau_L =\sum_{K\,  \in \cP_L} \tau_K\,.
$$
Notice that the sequence $\{\wt\tau_L\}_{L\in \mathcal{D}}$ form a Carleson sequence  (measure) with constant at most $C(m+1) Q^{\al}$.

Now we make a trick! We will estimate the right hand side  as
$$
R_L(\phi w) \le C(m+n+1)\La w\Ra_{\mu, L}^{-\al/2} \La \sigma\Ra_{\mu, L}^{-\al/2} \bigg(\sum_{K\,  \in \cP_L}  \La |\phi |w\Ra_{\mu, K}^p \frac{{\mu(K)}}{\mu(L)}\bigg)^{1/p} (\wt\tau_L)^{1/2}\,,
$$
where $p=2-\frac1{m+n+1}$. In fact,
$$
\bigg(\sum_{K\subset L, \,K\, is \,\, maximal}  \La | \phi |w\Ra_{\mu, K}^2 \frac{{\mu(K)}}{\mu(L)}\bigg)^{p/2}\le \sum_{K\,  \in \cP_L}  \La |\phi |w\Ra_{\mu, K}^p\bigg(\frac{{\mu(K)}}{\mu(L)}\bigg)^{p/2}\,.
$$
But if if $0\le j\le m$, then $(C^{-j})^{-\frac1{m+n+1}}\le C$, and therefore in the formula above $\bigg(\frac{{\mu(K)}}{\mu(L)}\bigg)^{1-\frac1{2(m+n+1)}} \le C\, \frac{{\mu(K)}}{\mu(L)}$, and $C$ depends only on the doubling constant of $\mu$. So the trick is justified.
Therefore, using Cauchy inequality, one gets
$$
R_L(\phi w) \le C(m+n+1)\La w\Ra_{\mu, L}^{-\al/2} \La \sigma\Ra_{\mu, L}^{-\al/2} \bigg(\sum_{K\,  \in \cP_L}  \La |\phi |^p w\Ra_{\mu, K}\La w\Ra_{\mu, K}^{p-1} \frac{{\mu(K)}}{\mu(L)}\bigg)^{1/p} (\wt\tau_L)^{1/2}\,.
$$
We can replace all $ \La w\Ra_{\mu, K}^{p-1}$ by $\La w\Ra_{\mu, L}^{p-1}$ paying the price by constant. This is again because all intervals larger than $K$ and smaller than $L$ will have there averages compared by constant at most $1\pm \frac1{m+n+1}$. And there are at most $m$ such intervals between $K$ and $L$. Finally,
\begin{equation}
\label{RL}
R_L(\phi w) \le C(m+n+1)\La w\Ra_{\mu, L}^{-\al/2} \La \sigma\Ra_{\mu, L}^{-\al/2} \bigg(\sum_{K\,  \in \cP_L}  \La |\phi |^p w\Ra_{\mu, K}\frac{{\mu(K)}}{\mu(L)}\bigg)^{1/p}\La w\Ra_{\mu, L}^{1-\frac1p}  (\wt\tau_L)^{1/2}
\end{equation}

We need the standard notations: if $\nu$ is an arbitrary positive measure we denote
$$
M_{\nu}f(x):=\sup_{r>0}\frac1{\nu(B(x,r))}\int_{B(x,r)} |f(x)|\,d\nu(x)\,.
$$
In particular $M_w$ will stand for this maximal function with $d\nu =w(x)\,d\mu$.

\bigskip

From \eqref{RL} we get
\begin{equation}
\label{RL1}
R_L(\phi w) \le C(m+n+1)\La w\Ra_{\mu, L}^{1-\al/2} \La \sigma\Ra_{\mu, L}^{-\al/2} \inf_L\, M_{w}(|\phi |^p )^{1/p} (\wt\tau_L)^{1/2}
\end{equation}

Now
\begin{equation}
\label{SR}
S_L(\psi\sigma)R_L(\phi w) \le C(m+n+1)\La w\Ra_{\mu, L}^{1-\al/2} \La \sigma\Ra_{\mu, L}^{1-\al/2}\frac{\inf_L \,M_{w}(|\phi |^p )^{1/p}}{\La \sigma\Ra_{\mu, L}^{1/2}} (\wt\tau_L)^{1/2}\sqrt{\sum_{J\subset L...} |(\psi \sigma, h^\sigma_J)|^2}\,,
\end{equation}
\begin{equation}
\label{RR}
R_L(\psi\sigma)R_L(\phi w) \le C(m+n+1)\La w\Ra_{\mu, L}^{1-\al} \La \sigma\Ra_{\mu, L}^{1-\al}\inf_L \,\,M_{w}(|\phi |^p )^{1/p}\inf_L\, M_{\sigma}(|\psi |^p )^{1/p} \wt\tau_L\,.
\end{equation}

Now we use the Carleson property of $\{\wt\tau_L\}_{L\in\cD}$. We need a simple folklore Lemma.

\begin{lemma}
\label{carl1}
Let $\{\al_L\}_{L\in\cD}$ define Carleson measure with intensity $B$ related to dyadic lattice $\cD$ on metric space $X$.  Let $F$ be a positive function on $X$. Then
\begin{equation}
\label{1}
\sum_L (\inf_L F)\, \al_L \le 2 B\int_{X} F\,d\mu\,.
\end{equation}
\begin{equation}
\label{2}
\sum_L\frac{ \inf_L F}{\La\sigma\Ra_{\mu, L}} \al_L \le C\,B\int_{X}\frac{F}{\sigma} d\mu\,.
\end{equation}
\end{lemma}

Now use \eqref{SR}. Then  the estimate of $III\le \sum_L S_L(\psi\sigma) R_L(\phi w)$ will be reduced to estimating
$$
(m+n+1)Q^{1-\al/2}\bigg(\sum_L \frac{\inf_L M_{w}(|\phi |^p )^{2/p}}{\La \sigma\Ra_{\mu, L}} \wt\tau_L\bigg)^{1/2}\le (m+n+1)^2\,Q\,\bigg(\int_{\R} (M_{w}(|\phi |^p ))^{2/p} wd\mu\bigg)^{1/2}
$$
$$
\le  (\frac1{2-p})^{1/p}(m+n+1)^2\,Q\,\bigg(\int_{\R} \phi^2 \, wd\mu\bigg)^{1/2}\le (m+n+1)^3\,Q\,\bigg(\int_{\R} \phi^2 \, wd\mu\bigg)^{1/2} \,.
$$
Here we used \eqref{2} and the usual estimates of maximal function $M_{\mu}$ in $L^q(\mu)$ when $q\approx 1$. Of course for $II$ we use the symmetric reasoning.

\bigskip

Now $IV$: we use \eqref{RR} first.
$$
\sum_L R_L(\psi\sigma) R_L(\phi w) \le (m+n+1)Q^{1-\al}\sum_L\inf_L\, M_{w}(|\phi |^p )^{1/p}\inf_L\, M_{\sigma}(|\psi |^p )^{1/p}\wt\tau_L
$$
$$
\le C (m+n+1)^2Q\int_{\R}(M_{w}(|\phi |^p ))^{1/p}\,(M_{\sigma}(|\psi |^p ))^{1/p}w^{1/2}\sigma^{1/2}d\mu
$$
$$
\le C (m+n+1)^2Q\,\bigg(\int_{\R}(M_{w}(|\phi |^p ))^{2/p}\, wd\mu\bigg)^{1/2}\,\bigg(\int_{\R}(M_{\sigma}(|\psi |^p ))^{2/p}\, \sigma d\mu\bigg)^{1/2}
$$
$$
\le C (m+n+1)^4\,Q\,\bigg(\int_{\R}\phi^2\, wd\mu\bigg)^{1/2}\bigg(\int_{\R}\psi^2\, \sigma d\mu\bigg)^{1/2}\,.
$$
Here we used \eqref{1} and the usual estimates of maximal function $M_{\mu}$ in $L^{2/p}(\mu)$ when $p\approx 2,\, p<2$.

\end{document}